\documentclass[preprint,a4paper,10pt]{elsarticle}
\usepackage[margin=3.5cm]{geometry}
\usepackage[T1]{fontenc}
\usepackage[utf8]{inputenc}
\usepackage[english]{babel}
\usepackage{amsmath,amssymb,amsthm}
\usepackage{graphicx}
\usepackage{xcolor}
\usepackage{microtype}
\usepackage[hidelinks]{hyperref}
\usepackage[nameinlink,noabbrev]{cleveref}
\usepackage{todonotes}

\newtheorem{theorem}{Theorem}
\newtheorem{lemma}[theorem]{Lemma}
\newtheorem{corollary}[theorem]{Corollary}
\newtheorem{proposition}[theorem]{Proposition}
\newtheorem{conjecture}[theorem]{Conjecture}
\theoremstyle{definition}
\newtheorem{definition}[theorem]{Definition}
\newtheorem{example}[theorem]{Example}
\theoremstyle{remark}
\newtheorem{remark}[theorem]{Remark}

\newcommand{\BWT}{\textsc{BWT}}

\newcommand{\RG}{RG}
\newcommand{\Z}{\mathbb{Z}}
\DeclareMathOperator{\DB}{DB}
\DeclareMathOperator{\GF}{\mathbb{F}}
\DeclareMathOperator{\ord}{ord}
\DeclareMathOperator{\SNF}{SNF}
\DeclareMathOperator{\Lyn}{\overline{N}}
\DeclareMathOperator{\Neck}{N}
\DeclareMathOperator{\InvNeck}{\hat{N}}
\DeclareMathOperator{\DBW}{DBW}
\DeclareMathOperator{\CM}{\mathcal{C}}

\begin{document}

\begin{frontmatter}

\title{Generalized De Bruijn Words, Invertible Necklaces, and the Burrows--Wheeler Transform\tnoteref{mfcs}}
\tnotetext[mfcs]{A preliminary version of this paper appeared in the proceedings of MFCS~2025~\cite{FiciGaboryMFCS2025}. In this extended version we added new results, all the details and proofs, and we updated the relevant literature.}

\author[unipa]{Gabriele Fici\fnref{fici-support}}
\ead{gabriele.fici@unipa.it}
\ead[url]{https://orcid.org/0000-0002-3536-327X}
\fntext[fici-support]{Supported by MIUR project PRIN 2022 APML -- 20229BCXNW.}
\address[unipa]{Dipartimento di Matematica e Informatica, Università di Palermo, Palermo, Italy}

\author[uwr]{Estéban Gabory\fnref{gabory-support}}
\ead{estebangabory@gmail.com}
\ead[url]{https://orcid.org/0000-0002-9897-1512}
\fntext[gabory-support]{Supported by the Polish National Science Centre grant number 2023/51/B/ST6/01505.}
\address[uwr]{Institute of Computer Science, University of Wrocław, Wrocław, Poland}

\begin{abstract}
We define generalized de Bruijn words as those words having a Burrows--Wheeler transform that is a concatenation of permutations of the alphabet. We show that generalized de Bruijn words are in one-to-one correspondence with Hamiltonian cycles in the generalized de Bruijn graphs, introduced in the early 1980s in the context of network design.
We prove that generalized de Bruijn words retain the asymptotically maximal
factor complexity of ordinary de Bruijn words over a fixed
alphabet.  We then use this class to construct an infinite binary family on
which the Burrows--Wheeler transform reduces this complexity from asymptotically
maximal to the constant $2$.  In particular, this shows that the BWT can
substantially improve dictionary compressibility.
When the size of the alphabet is a prime $p$, we define invertible necklaces as those whose BWT-matrix is nonsingular. We show that invertible necklaces of length $n$ correspond to normal bases of the finite field $\GF_{p^n}$, and that they form an abelian group isomorphic to the Reutenauer group $\RG_p^n$.
Using known results in abstract algebra, we can establish a connection between generalized de Bruijn words and invertible necklaces. In particular, we highlight a correspondence between binary de Bruijn words of order $d+1$, binary necklaces of length $2^{d}$ having an odd number of $1$'s, invertible BWT matrices of size $2^{d}\times 2^{d}$, and normal bases of the finite field $\GF_{2^{2^{d}}}$.
\end{abstract}


\end{frontmatter}

\section{Introduction}
\label{sec:intro}

The Burrows--Wheeler matrix of a word $w$ is the matrix whose rows are the conjugates (rotations) of $w$ in ascending lexicographic order. Its last column is called the  Burrows--Wheeler transform of $w$~\cite{BWT}, and has the property of being easier to compress when the input word is a repetitive text. For this reason, it is largely used in textual data compression, in particular in bioinformatics~\cite{Li2010-ih,Langmead2009-zd,Langmead2012-wd}. But the  Burrows--Wheeler transform also has many interesting combinatorial properties (see~\cite{DBLP:conf/cie/RosoneS13,bwtsurvey} for a survey). Since the Burrows--Wheeler transform is the same for any conjugate of $w$, it can be viewed as a map from the set of necklaces (conjugacy classes of words) to the set of words.

A \emph{de Bruijn word} of order $d$ over an alphabet $\Sigma$ is a necklace of length $|\Sigma|^d$ such that each of the $|\Sigma|^d$ distinct words of length $d$ occurs exactly once in it, as a cyclic factor. The binary case was studied in de Bruijn's foundational paper~\cite{deBruijn1946}. Since each of the $|\Sigma|^{d-1}$ distinct words of length $d-1$ occurs as a prefix of $|\Sigma|$ consecutive rows of the Burrows--Wheeler matrix of a de Bruijn word, the  Burrows--Wheeler transform of a de Bruijn word is characterized (among words that are images under the Burrows--Wheeler transform) by the fact that it is a concatenation of $|\Sigma|^{d-1}$ alphabet-permutations (de Bruijn words of order $1$). This motivates us to define a \emph{generalized de Bruijn word} as one whose Burrows--Wheeler transform is a concatenation of (any number of) alphabet-permutations.

As is well known, de Bruijn words are in $1$-to-$1$ correspondence with Hamiltonian cycles in de Bruijn graphs. We show that, analogously,  generalized de Bruijn words are in $1$-to-$1$ correspondence with Hamiltonian cycles in \emph{generalized de Bruijn graphs}, introduced in the early 1980s independently by Imase and Itoh~\cite{DBLP:journals/tc/ImaseI81}, and by Reddy, Pradhan, and Kuhl~\cite{Reddy} (see also~\cite{DBLP:journals/networks/DuH88}) in the context of network design. The generalized de Bruijn graph $\DB(k,n)$  has vertices $\{0,1,\ldots,n-1\}$ and for every vertex $m$ there is an edge from $m$ to $km+i \mod (n)$ for every $i=0,1,\ldots,k-1$.

In particular, we provide a new and simple interpretation of the well-known correspondence between de Bruijn words and Hamiltonian cycles in de Bruijn graphs in terms of the \emph{inverse standard permutation} of the  Burrows--Wheeler transform, and we show that this interpretation extends to the generalized case.

A natural requirement for a meaningful generalization of de Bruijn words is
that it preserve their combinatorial richness.  We formalize this requirement
using factor complexity and prove that generalized de Bruijn words,
like ordinary de Bruijn words, have asymptotically maximal substring
complexity over a fixed alphabet.

This first result also provides a class of highly complex inputs with which to
study the compression effect of the BWT.  Kociumaka, Navarro, and
Prezza~\cite{KociumakaNavarroPrezza2023} introduced the factor complexity
$\delta$ as a lower bound on several dictionary-compression measures.
Bannai, I, and Nakashima~\cite{BannaiINakashima2025} constructed an infinite
family on which the Bijective BWT (BBWT), a bijective variant of the BWT,
reduces $\delta$ from asymptotically maximal to constant, naturally raising
the analogous question for the ordinary BWT.  Using generalized de Bruijn
words, we answer this question affirmatively: we construct an infinite binary
family for which the ordinary BWT reduces $\delta$ to $2$.  Thus, beyond the
particular class used in the construction, the result shows that the BWT
itself can substantially improve dictionary compressibility.

The other class of words we introduce is that of \emph{invertible necklaces}. We call an aperiodic necklace over an alphabet of prime size $p$ invertible if its Burrows--Wheeler matrix is nonsingular, i.e., has determinant nonzero modulo $p$. We show that each invertible necklace of length $n$, as a set of vectors over the base field $\GF_p$, corresponds to a normal basis of the field $\GF_{p^n}$.
Invertible necklaces can also be represented by conjugacy classes of their circulant matrices, and they form an abelian group, called the \emph{Reutenauer group} $\RG_p^{n}$~\cite{DuzhinJMS}. The set of invertible necklaces of a given length $n$ can therefore be endowed with a multiplication operation that makes it isomorphic to the $n$-th Reutenauer group.

Using a known result in abstract algebra~\cite{CTR01}, we show that every aperiodic necklace of length $n$ with nonzero weight modulo $p$ (the weight of a word is the sum of its digits) is invertible if and only if $n$ is either a power of $p$ or a $p$-rooted prime. However, it is an open problem to decide whether there are infinitely many lengths $n$, different from a power of $p$, for which every aperiodic necklace of length $n$ with nonzero weight modulo $p$ has an invertible Burrows--Wheeler matrix. This seems to be a challenging problem, as it is related to Artin's conjecture on primitive roots.

In the last part of the paper, we show a connection between generalized de Bruijn words and invertible necklaces.
Chan, Hollmann and Pasechnik~\cite{ECCGTA13} proved that for every prime $p$ and any $n$, one has  $\RG_p^{n} \cong K(\DB(p,n)) \oplus \Z_{p-1}$, where $K(G)$ denotes the sandpile group of the graph $G$. Since the structure of the sandpile group of an Eulerian graph is determined by the invariant factors of its Laplacian matrix, Chan et al.~gave the precise structure of sandpile groups of generalized de Bruijn graphs, and hence of Reutenauer groups.

In particular, when $p=2$ and $n=2^d$, we show that there is a bijection between de Bruijn words of order $d+1$ and binary necklaces of length $2^d$ with an odd number of $1$'s. %

A preliminary version of this paper appeared in the proceedings of
MFCS~2025~\cite{FiciGaboryMFCS2025}.  The present version additionally proves
that generalized de Bruijn words have asymptotically maximal substring
complexity and uses them to exhibit an infinite family on
which the BWT reduces it from asymptotically maximal to constant
(Section~\ref{sec:factors}).

\paragraph*{Our contributions}

We introduce two new classes of circular words: generalized de Bruijn words and invertible necklaces, both of which are defined in terms of the BWT. We show that generalized de Bruijn words correspond to those whose BWT has an inverse standard permutation that labels a Hamiltonian cycle in a generalized de Bruijn graph (Theorem~\ref{thm:charGen}). This result allows us to derive a formula for counting the number of generalized de Bruijn words using the BEST theorem (Theorem~\ref{thm:counting}).

We next ask whether generalized de Bruijn words retain the factor richness of
ordinary de Bruijn words.  In terms of cyclic factor complexity
$\delta$, we prove that every generalized de Bruijn word $[w]$ of length $kn$
over $\Sigma_k$ satisfies
$$
    \delta([w])=\Theta\!\left(\frac{kn}{\log_k(kn)}\right),
$$
which is asymptotically maximal for fixed $k$ (Theorem~\ref{thm:delta-gdbw}).
We then use this result to obtain a separation for the ordinary BWT: we
construct an infinite family of binary generalized de Bruijn words $[u_d]$
for which $\delta([u_d])=\Theta(3^d/d)$, while
$\delta(\BWT([u_d]))=2$ (Theorem~\ref{thm:delta-two-bwt}).  This answers, for
the standard BWT, the question naturally suggested by the corresponding
separation for the Bijective BWT~\cite{BannaiINakashima2025}, and shows more
generally that the BWT can substantially improve dictionary compressibility.

Next, we focus on alphabets of prime size $p$. We define invertible necklaces as those whose Burrows–Wheeler matrix is nonsingular. We demonstrate that the conjugacy classes of invertible necklaces form an abelian group isomorphic to the Reutenauer group of invertible circulant matrices, and characterize them in terms of normal elements of finite fields (Theorem~\ref{thm:invneckeq}).  We also characterize the lengths for which every aperiodic necklace of nonzero weight modulo $p$ is invertible (Corollary~\ref{cor:ArtinBWT}). %

Finally, we connect generalized de Bruijn words and invertible necklaces using a previously known isomorphism between the sandpile groups of generalized de Bruijn graphs and Reutenauer groups.  We obtain a counting formula over alphabets of prime size and, in the binary case, this shows that generalized de Bruijn words, invertible necklaces, and normal bases of finite fields are in bijection.%(Theorems~\ref{thm:final1} and~\ref{thm:final2})

\paragraph*{Related work}

Several generalizations of de Bruijn words (sequences) have been proposed in the literature~\cite{Blum83,DBLP:journals/dm/Au15,DBLP:conf/cpm/NelloreW22}, all based on the number of occurrences of factors (substrings) of some kind.

The relation between the BWT and de Bruijn words has been studied in several papers. The extension of the BWT to finite multisets of necklaces was introduced by Mantaci et al.~\cite{MantaciRestivoRosoneSciortino2007}, based on the Gessel--Reutenauer bijection~\cite{GesselReutenauer1993}. Higgins~\cite{DBLP:journals/tcs/Higgins12} subsequently used this extension to characterize de Bruijn sets---generalizations of de Bruijn words---by their BWT image, recovering a BWT-based characterization of de Bruijn words in the singleton case. In~\cite{DBLP:conf/latin/LiptakP24}, this correspondence is used to design an efficient algorithm for constructing arbitrary de Bruijn words. More recently, Sawada and Gabri{\'c} gave practical Las Vegas algorithms for generating de Bruijn sequences uniformly at random~\cite{SawadaGabric2025}. Fici, Gabory, Romana, and Sciortino~\cite{FiciGaboryRomanaSciortino2026} used generalized de Bruijn words to study totally unclustered BWT images over non-binary alphabets and their connection with Artin's conjecture.

Kociumaka, Navarro, and Prezza~\cite{KociumakaNavarroPrezza2023}
introduced the factor complexity $\delta$ as a lower bound on
several measures associated with dictionary compression.  Bannai, I, and
Nakashima~\cite{BannaiINakashima2025} subsequently constructed a family that
is asymptotically maximally incompressible according to these measures but
becomes highly compressible after applying the Bijective BWT.
Their result naturally raises the analogous question for the standard BWT.
Our construction answers it affirmatively, using generalized de Bruijn words
as the factor-rich input family.

Recently, the algebraic structures generated by invertible BWT matrices have also been explored. In~\cite{AAM2026_103081,ReSh26}, the authors showed that BWT matrices of Christoffel words are closed under multiplication and described the multiplicative group of invertible BWT matrices of Christoffel words.
Christoffel words are the unbordered factors of aperiodic infinite words of minimal factor complexity (Sturmian words)~\cite{Book08}. The conjugacy classes of Christoffel words are precisely the binary aperiodic necklaces whose BWT has the minimal possible number of equal-letter runs~\cite{DBLP:journals/ipl/MantaciRS03}.

\paragraph*{Paper organization}

Our paper is structured as follows. In Section~\ref{sec:prel}, we introduce the necessary preliminaries on words, necklaces, and the Burrows–Wheeler transform. Section~\ref{sec:db} defines generalized de Bruijn words and establishes their graph-theoretic characterization using generalized de Bruijn graphs, which extends the classical correspondence between de Bruijn words and de Bruijn graphs. Section~\ref{sec:factors} studies the factor complexity of generalized de Bruijn words and of their Burrows--Wheeler transforms. In Section~\ref{sec:invertible}, we introduce the concept of an invertible necklace and relate it to circulant matrices and Reutenauer groups. Section~\ref{sec:final} connects these concepts, using known group isomorphisms to show a bijection between generalized de Bruijn words and invertible necklaces in the context of prime-sized alphabets. Finally, in Section~\ref{sec:conclusion}, we summarize our findings and propose open problems and future directions.

\section{Preliminaries}\label{sec:prel}

We begin by introducing some preliminary definitions %
related to strings and words. For a thorough introduction, we refer the reader to \cite{crochemore_algorithms_2007}, and \cite{lothaire_combinatorics_1983}.

Let $\Sigma_k=\{0,1,\ldots,k-1\}$, $k>1$, be a sorted set of \emph{letters}. 
  
A \emph{word} over the alphabet $\Sigma_k$ is a concatenation of elements of $\Sigma_k$. The \emph{length} of a word $w$ is denoted by $|w|$. %
 For a letter $i\in\Sigma_k$, $|w|_i$ denotes the number of occurrences of $i$ in $w$. The vector $(|w|_{0},\ldots,|w|_{k-1})$ is the \emph{Parikh vector} of $w$.

Let  $w=w_0\cdots w_{n-2}w_{n-1}$ be a word of length $n>0$. The \emph{weight} of $w$ is  $\sum_{j=0}^{n-1}w_j$. When $n>1$, the \textit{shift} of $w$ is the word $\sigma(w)=w_{n-1}w_0\cdots w_{n-2}$.

For a word $w$, the \emph{$n$-th power} of $w$ is the word $w^n$ obtained by concatenating $n$ copies of $w$. %

A \emph{factor} (or \emph{substring}) of a word $w$ is a word $x$ such that
$w=uxv$ for some words $u$ and $v$.

\begin{definition}
    We call an \emph{alphabet-permutation}  a word over $\Sigma_k$ that contains each letter of $\Sigma_k$ exactly once, and an \emph{alphabet-permutation power} a concatenation of one or more alphabet-permutations over $\Sigma_k$.
\end{definition}

\begin{example}
    The word $w=210201102102120$ is an alphabet-permutation power over $\Sigma_3$.
\end{example}

Notice that an alphabet-permutation power of length $n$ has a balanced Parikh vector, i.e., a Parikh vector of the form $(\frac{n}{k},\frac{n}{k},\ldots,\frac{n}{k})$.

In the binary case, a word $w$ is an alphabet-permutation power if and only if $w=\tau(v)$, where $v$ is a binary word of length $|w|/2$ and $\tau$ is the \emph{Thue--Morse morphism}, the substitution that maps $0$ to $01$ and $1$ to $10$.

Two words $w$ and $w'$ are \emph{conjugates} if $w=uv$ and $w'=vu$ for some words $u$ and $v$. The conjugacy class of a word $w$ can be obtained by repeatedly applying the shift operator, and contains $|w|$ distinct elements if and only if $w$ is \emph{primitive}, i.e., for any nonempty word $v$ and positive integer $n$, $w= v^n$ implies $n=1$. %
A \emph{necklace} (resp.~\emph{aperiodic necklace}) $[w]$ is a conjugacy class of words (resp.~of \emph{primitive words}). Necklaces are also called \emph{circular words}. Notice that, given a word $w$, the weight of $w$ is invariant under shift, hence we can define the \emph{weight} $wt([w])$ of the necklace $[w]$ as the weight of any of its representatives.

A \emph{cyclic factor} of a necklace $[w]$ is a factor of $ww$ of length at
most $|w|$.  This definition does not depend on the representative $w$ of the
necklace.

For example, $[1100]=\{1100,0110,0011,1001\}$ is an aperiodic necklace, while $[1010]=\{1010,0101\}$ is not aperiodic.

The number of aperiodic necklaces of length $n$ over $\Sigma_k$ is given by
\[\Lyn(k,n)=\frac{1}{n}\sum_{d | n}\mu\left(\frac{n}{d}\right)k^d\]
where $\mu(n)$ is the M\"obius function, defined by: $\mu(1)=1$, $\mu(n)=(-1)^j$ if $n$ is the product of $j$ distinct primes or $0$ otherwise, i.e., if $n$ is divisible by the square of a prime number.

The number of  necklaces of length $n$ over $\Sigma_k$ is given by
\[\Neck(k,n)=\sum_{d|n}\Lyn(k,d)=\frac{1}{n}\sum_{d | n}\phi\left(\frac{n}{d}\right)k^d\]
where $\phi(n)$ is the Euler's totient function. Recall that the Euler's totient function $\phi(n)$ counts the number of positive integers less than or equal to $n$ and coprime with $n$, i.e., the order of the multiplicative group $\mathbb{Z}_n^*$, and can be computed using the formula
\[\phi(n)=n\prod_{p|n}\left(1-\dfrac{1}{p}\right)\]
for $n>1$, and $\phi(1)=1$, where the product is taken over the distinct primes dividing $n$.

\bigskip

 The \textit{Burrows--Wheeler matrix} (BWT matrix) of a necklace $[w]$ is the matrix whose rows are the $|w|$ shifts of $w$ in ascending\footnote{This is the standard convention in the literature. Of course, one can also choose the descending lexicographic order, the properties are symmetric.} lexicographic order. Let us denote by $F$ and $L$, respectively, the first and the last column of the BWT matrix of $[w]$. We have that $F=0^{n_0}1^{n_1}\cdots (k-1)^{n_{k-1}}$, where $(n_0,\ldots, n_{k-1})$ is the Parikh vector of $[w]$; $L$, instead, is the \emph{Burrows--Wheeler Transform} (BWT) of $[w]$. The BWT is therefore a map from the set of necklaces to the set of words. As is well known, it is an injective map (we describe below how to invert it). A word is a \emph{BWT image} if it is the BWT of some necklace. 

The \emph{standard permutation}  of a word $u=u_0u_1\cdots u_{n-1}$,  $u_i\in\Sigma_k$, is the permutation $\pi_u$ of $\{0,1,\ldots,n-1\}$ such that $\pi_u(i)<\pi_u(j)$ if and only if $u_i<u_j$ or  $u_i=u_j$ and $i<j$. In other words, in one-line notation, $\pi_u$ orders distinct letters of $u$ lexicographically, and equal letters by occurrence order, starting from $0$. %
When $u$ is the BWT image of some necklace, the standard permutation is also called $LF$-mapping. As is well known, a word $u$ is a BWT image of an aperiodic necklace if and only if $\pi_{u}$ is a length-$n$ cycle.

The \emph{inverse standard permutation} $\pi^{-1}_u$ of a word $u$, written in one-line notation, can be obtained by listing in left-to-right order the positions of $0$ in $u$, then the positions of $1$, and so on. The inverse standard permutation of a BWT image is also called $FL$-mapping.

An aperiodic necklace can be uniquely reconstructed from its BWT, using either the standard permutation or its inverse. To do this, it is convenient to write these permutations in cycle form.
Let $L$ be the BWT of $[w]$, and $\pi_L=(j_0\, j_1\, \cdots \, j_{n-1})$ be the standard permutation of $L$ in cycle form, with $j_0=0$. Then, $L_{j_{n-1}}L_{j_{n-2}}\cdots L_{j_0}$ is the first row of the BWT matrix of $[w]$, i.e., the lexicographically least element of the necklace $[w]$. This is also equal to $F_{i_0}F_{i_1}\cdots F_{i_{n-1}}$, where $\pi^{-1}_L=(i_0\, i_1\, \cdots \, i_{n-1})$ is the inverse standard permutation of $L$ in cycle form,  with $i_0=0$. %

The following remark will be used in later sections:
\begin{remark}\label{rmk:balancedLF}
    In the case of a balanced Parikh vector $(\frac{n}{k},\frac{n}{k},\ldots,\frac{n}{k})$, one can reverse the BWT from its inverse standard permutation in cycle form $\pi_u^{-1}=(i_0\, i_1\, \cdots \, i_{n-1})$, $i_0=0$, by mapping $i_{\ell}$ to  $\left \lfloor \dfrac{i_{\ell}}{n/k}\right \rfloor$ for $\ell=0,1,\ldots,n-1$. 
\end{remark}

For example, the BWT  of $[w]=[220120011]$ is the word $u=202001121$, whose 
 inverse standard permutation in cycle form is $\pi_u^{-1}=(0\, 1\, 3\, 5\, 8\, 7\, 2\, 4\, 6)$. Mapping: $0,1,2$ to $0$; $3,4,5$ to $1$; $6,7,8$ to $2$, one obtains the word $001122012$, the first row of the BWT matrix of $[w]$.

Necklaces that are not aperiodic can also be reconstructed from their BWT, as a consequence of the following result:

\begin{proposition}[{\cite[Proposition~2]{DBLP:journals/ipl/MantaciRS03}}]\label{nonprimitive}
\label{le:bwtpower}
For any $c\geq 1$, $u=u_1u_2\cdots u_n$ is the BWT of an aperiodic necklace $[w]$ if and only if $u_1^cu_2^c\cdots u_n^c$ is the BWT of $[w^c]$.
\end{proposition}

 Let $ G = (V, E)$ be a finite directed graph, which may have loops and multiple edges (in this case, it is also called a multigraph in the literature). Each edge
$e \in E$ is directed from its source vertex $s(e)$ to its target vertex $t(e)$. 

The directed \emph{line graph} (or \emph{edge graph}) $\mathcal{L} G =(E, E')$ of $G$ has as vertices the edges of $G$, and as edges the set
$$E' = \{(e_1, e_2) \in E \times E \mid s(e_2) = t(e_1)\}.$$

An oriented \emph{spanning tree} of $G$ is a subgraph containing all of the vertices of $G$, having no directed cycles, in which one vertex, the root, has outdegree $0$, and every other vertex has outdegree $1$. The number $\kappa(G)$ of oriented spanning trees of $G$ is sometimes called the \textit{complexity} of $G$.%

A directed graph is \emph{Hamiltonian} if it has a \emph{Hamiltonian cycle}, i.e., one that traverses each node exactly once, while it is \emph{Eulerian} if it has an \emph{Eulerian cycle}, i.e., one that traverses each edge exactly once. As is well known, a directed graph is Eulerian if and only if it is connected and $indeg(v) = outdeg(v)$ for all vertices $v$. If a graph is Eulerian, its line graph is Hamiltonian.

\section{Generalized de Bruijn graphs and generalized de Bruijn words}\label{sec:db}

We start by briefly discussing (ordinary) de Bruijn graphs and de Bruijn words, and relating them to the inverse standard permutation of the Burrows--Wheeler transform.

Recall that the de Bruijn graph $\DB(\Sigma_k,k^d)$ of order $d$ is the directed graph whose vertices are the words of length $d$ over $\Sigma_k$ and there is an edge from $a_iw$ to $v$, where $a_i$ is a letter, if and only if $v=wa_j$ for some letter $a_j$. As is well known, de Bruijn graphs are Eulerian and Hamiltonian. One has $\mathcal{L}\DB(\Sigma_k,k^d) = \DB(\Sigma_k,k^{d+1})$. 

A \emph{de Bruijn word} of order $d$ over $\Sigma_k$ is a necklace over $\Sigma_k$ containing as a cyclic factor each of the $\Sigma_k^d$ words of length $d$ over $\Sigma_k$ exactly once. De Bruijn words correspond to Hamiltonian cycles in $\DB(\Sigma_k,k^d)$, as they can be obtained by concatenating the first letter of the labels of the vertices it traverses.

Now, since the vertex labels of $\DB(\Sigma_k,k^d)$ are the base-$k$ representations of the integers in $\{0,1,\ldots,k^d-1\}$ on $d$ digits (padded with leading zeroes), the de Bruijn graph of order $d$ over $\Sigma_k$ can be equivalently defined as the one having vertex set $\{0,1,\ldots,k^d-1\}$ and containing an edge from $m$ to $km+i \mod (k^d)$ for every $i=0,1,\ldots,k-1$. We let $\DB(k,k^d)$ denote this equivalent representation of the de Bruijn graph of order $d$.
A Hamiltonian cycle in $\DB(k,k^d)$ is then a cyclic permutation of $\{0,1,\ldots,k^d-1\}$. 
The relation between a de Bruijn word of order $d$ over $\Sigma_k$, i.e., a  Hamiltonian cycle in $\DB(\Sigma_k,k^d)$, and the corresponding Hamiltonian cycle in $\DB(k,k^d)$, is given in the following lemma.

\begin{lemma}\label{lem:dbw as FL}
Let $[w]=[w_0\cdots w_{k^d-1}]$ be the necklace encoding a Hamiltonian cycle in $\DB(\Sigma_k,k^d)$, and let $x_i\in\Sigma_k^d$ be the label of the vertex visited at position $i$, so that $w_i$ is the first letter of $x_i$. Let $[u]=[u_0\cdots u_{k^d-1}]$ be the corresponding Hamiltonian cycle in $\DB(k,k^d)$, where $x_i$ is the $k$-ary expansion of $u_i$ on $d$ digits for each $i=0,\ldots,k^d-1$. Then $[u]$ is the inverse standard permutation of the Burrows--Wheeler transform of $[w]$, in cycle form. In particular, one has $w_i=\lfloor u_i/k^{d-1} \rfloor$ for every $0\leq i\leq k^d-1$.
\end{lemma}

\begin{proof}
For a given integer $n$ in $\{0,\ldots,  k^d-1\}$, the first digit of $n$ in its base-$k$ representation on $d$ digits is $\lfloor \frac{n}{k^{d-1}} \rfloor$. The claim then follows directly from Remark~\ref{rmk:balancedLF} %
and the fact that de Bruijn words have a balanced Parikh vector.
\end{proof}

\begin{example}
In $\DB(\Sigma_2,8)$, the de Bruijn word $[w]=[00010111]$ corresponds to the Hamiltonian cycle obtained by visiting the nodes $000$, $001$, $010$, $101$, $011$, $111$, $110$, $100$; and to the Hamiltonian cycle $(0,1,2,5,3,7,6,4)$ in $\DB(2,8)$. Indeed, one has $\BWT([w])=10011010$, and $\pi^{-1}_{\BWT([w])}=(0\, 1\, 2\, 5\, 3\, 7\, 6\, 4)$, in cycle notation. By applying the mapping $i\mapsto \lfloor\frac{i}{4}\rfloor$, one retrieves $[w]=[00010111]$.
\end{example}

We now introduce generalized de Bruijn words.  In~\cite{DBLP:journals/tcs/Higgins12}, Higgins showed that a word $w$ of length $k^n$ over $\Sigma_k$  is a (ordinary) de Bruijn word if and only if its BWT is an alphabet-permutation power. This motivates us to introduce the following definition:

\begin{definition}
A necklace of length $kn$ over $\Sigma_k$ is a \emph{generalized de Bruijn word} if its BWT is an alphabet-permutation power over $\Sigma_k$.
\end{definition}

For example, $[02201331]$ is a generalized de Bruijn word since its BWT is $21302031$.

Notice that, in the literature, there already exist several other definitions of generalized de Bruijn words~(see, e.g.,~\cite{Blum83,DBLP:journals/dm/Au15,DBLP:conf/cpm/NelloreW22}).

The \emph{generalized de Bruijn graph} $\DB(k,n)$ has vertices $\{0,1,\ldots,n-1\}$ and for every vertex $m$ there is an edge from $m$ to $km+i \mod (n)$ for every $i=0,1,\ldots,k-1$. Generalized de Bruijn graphs have been introduced independently by Imase and Itoh~\cite{DBLP:journals/tc/ImaseI81}, and by Reddy, Pradhan, and Kuhl~\cite{Reddy} (see also~\cite{DBLP:journals/networks/DuH88}). They are Eulerian and Hamiltonian. By definition, a generalized de Bruijn graph $\DB(k,n)$ is an ordinary de Bruijn graph when $n=k^d$ for some $d>0$. As an example, the generalized de Bruijn graph $\DB(3,6)$ is displayed in Fig.~\ref{fig:DB36}.

\begin{figure}
    \centering
  \includegraphics[width=11cm]{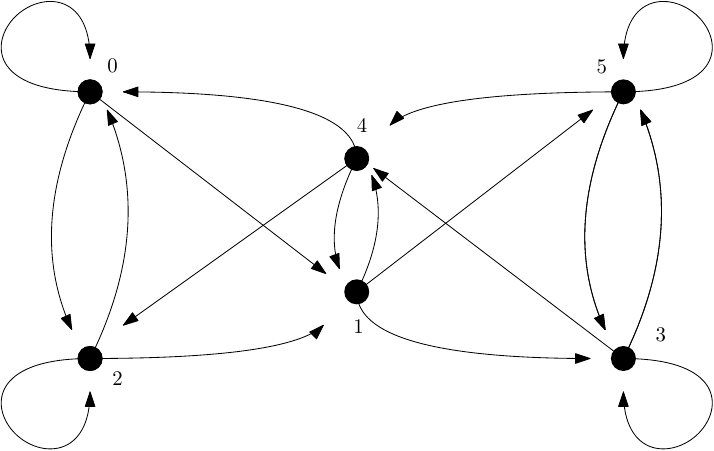}
    \caption{The generalized de Bruijn graph $\DB(3,6)$. }
    \label{fig:DB36}
\end{figure}

The line graph of $\DB(k,n)$ is $\DB(k,kn)$~\cite{DBLP:journals/dm/LiZ91}. %
The Eulerian cycles of $\DB(k,n)$ correspond to the Hamiltonian cycles of $\DB(k,kn)$.%

In Theorem~\ref{thm:charGen}, we give a characterization of generalized de Bruijn words in terms of generalized de Bruijn graphs, which generalizes Lemma~\ref{lem:dbw as FL}.

\begin{lemma}\label{lem:APW-char}
Let $u$ be a word of length $kn$ over $\Sigma_k$, with Parikh vector $(n,n,\ldots, n)$. The following statements are equivalent:
\begin{enumerate}
    \item\label{APW-char1} The word $u$ is an alphabet-permutation power;
    \item\label{APW-char2} For each $0\le j< k$, $0\le i< n$, one has $\pi^{-1}_u(i+jn)\in[ki,k(i+1)-1]$; %
   \item\label{APW-char3} For each $0\le \ell< kn$, there is an edge from $\ell$ to $\pi^{-1}_u(\ell)$ in $\DB(k,kn)$. 
\end{enumerate}
\end{lemma}

\begin{proof}
We first observe that $(\ref{APW-char3})\iff (\ref{APW-char2})$. By definition, the edges of the de Bruijn graph $\DB(k, kn)$ are of the form $(\ell, k\ell + i' \bmod kn)$ for $0 \le i' < k$. In particular, for any $\ell = i + jn$ with $0 \le i < n$ and $0 \le j < k$, we have $k(i + jn) + i' \bmod kn \in [ki, k(i+1) - 1]$. Thus, $\pi_u^{-1}(\ell)$ lies in this interval if and only if there is an edge from $\ell$ to $\pi_u^{-1}(\ell)$ in $\DB(k, kn)$. %
Next, we show that (\ref{APW-char1}) implies (\ref{APW-char2}). We start with the following general remark: for a word $u$ and an integer $0\le \ell\leq |u|-1$, one has $\pi^{-1}_u(\ell)=p$ if $p$ is the position of the $\ell$th letter in $u$, according to the order defined by $\pi^{-1}_u$, which is the lexicographic order with equal letters sorted by occurrence position. Suppose $u$ is an alphabet-permutation power, so $u$ consists of $n$ consecutive blocks of length $k$, each of which is a permutation of $\Sigma_k$. In particular, each letter appears exactly $n$ times in total. Let $i < n$ and $j < k$. The index $\pi^{-1}_u(i+jn)$ corresponds to the $(i+1)$th occurrence of the letter $j$ in $u$. Since $u$ is a concatenation of $n$ blocks, and each block is a permutation of $\Sigma_k$, the $(i+1)$th occurrence of letter $j$ must appear within the $(i+1)$th block of $u$, i.e., within positions $[ki, k(i+1)-1]$. Therefore, $\pi_u^{-1}(i + jn) \in [ki, k(i+1)-1]$, as required. This proves that $(\ref{APW-char1})\implies (\ref{APW-char2})$.

It remains to show that (\ref{APW-char2}) implies (\ref{APW-char1}). Assume that for all $0 \le j < k$ and $0 \le i < n$, we have $\pi_u^{-1}(i + jn) \in [ki, k(i+1) - 1]$. Fix any $i < n$. Then the $k$ values $\pi_u^{-1}(i + jn)$ for $j = 0, \dots, k-1$ all lie in $[ki, k(i+1) - 1]$. Since $\pi_u^{-1}$ is a permutation, these values must be distinct. Therefore, $\{\pi_u^{-1}(i), \pi_u^{-1}(n + i), \dots, \pi_u^{-1}((k - 1)n + i)\} = [ki, k(i+1) - 1]$.

Let us now examine the content of these positions in $u$. The value $u[\pi_u^{-1}(i + jn)]$ is the $j$-th letter in this collection. But $\pi^{-1}_u(i+jn)$ is the position of the $(i+1)$th occurrence of the letter $j$, hence $u[\pi_u^{-1}(i + jn)] = j$. It follows that $\{u[\pi_u^{-1}(i)], u[\pi_u^{-1}(n+i)], \dots, u[\pi_u^{-1}((k-1)n + i)]\} = \Sigma_k$, which means that the substring $u[ki\, ..\, k(i+1)-1]$ is a permutation of $\Sigma_k$. Since this holds for every $0 \le i < n$, the word $u$ is a concatenation of $n$ permutations of $\Sigma_k$, i.e., an alphabet-permutation power. This completes the proof.
\end{proof}

As a consequence, we have:

\begin{theorem}\label{thm:charGen}
 A necklace is the label of a Hamiltonian cycle of a generalized de Bruijn graph $\DB(k,kn)$ if and only if it is the inverse standard permutation, in cycle form, of the BWT of a generalized de Bruijn word of length $kn$ over $\Sigma_k$.
\end{theorem}

\begin{proof}
The fact that every Hamiltonian cycle $q$ of $\DB(k,kn)$ corresponds to the inverse standard permutation of some length $kn$ word over $k$ letters with Parikh vector $(n,\ldots,n)$ follows from condition (\ref{APW-char2}) of Lemma~\ref{lem:APW-char} and from the fact that a permutation $\pi$ is the inverse standard permutation of a word over $k$ letters if and only if there are at most $k-1$ indices $i$ such that $\pi(i+1)<\pi(i)$~\cite[Theorem 1]{DBLP:journals/tcs/CrochemoreDP05}. To see that this word is balanced, fix $0\leq j<k$. The edge condition gives $q(jn+i)\in[ki,k(i+1)-1]$, so $q(jn),\ldots,q(jn+n-1)$ is the increasing block of positions assigned to the letter $j$. This block has length $n$, and therefore every letter occurs exactly $n$ times.
We use the fact that generalized de Bruijn words are aperiodic necklaces: indeed, if $w = u^r$ is a power (with $r > 1$), then its BWT can be derived from the BWT of $u$ by repeating each letter $r$ times~(Proposition~\ref{le:bwtpower}). Consequently, the BWT of a periodic necklace cannot be an alphabet-permutation power and, conversely, a periodic necklace cannot have a BWT that is an alphabet-permutation power. We now conclude from Lemma~\ref{lem:APW-char}, and from the fact that a word is a BWT image of an aperiodic necklace if and only if its (inverse) standard permutation is a cycle.
 \end{proof}

Therefore, the generalized de Bruijn words of length $kn$ over the alphabet $\Sigma_k$ can be constructed from the  Hamiltonian cycles of the generalized de Bruijn graph $\DB(k,kn)$ by mapping $i\mapsto \lfloor i/n\rfloor$ for every $i=0,\ldots,  kn-1$.

\begin{example}
    The Hamiltonian cycles in $\DB(3,6)$ are: $(0,1,3,5,4,2)$, $(0,1,5,3,4,2)$, $(0,2,1,3,5,4)$, and $(0,2,1,5,3,4)$. Applying the mapping $i\mapsto \lfloor i/2\rfloor$, one obtains the ternary generalized de Bruijn words of length $6$:  $[001221]$, $[002121]$, $[010122]$, and $[010212]$.
\end{example}

Thanks to the characterization given in Theorem~\ref{thm:charGen}, we can obtain a formula for counting generalized de Bruijn words. We make use of the BEST theorem for directed graphs, together with the fact that the number of Hamiltonian cycles in $\DB(k,kn)$ is equal to the number of Eulerian cycles in $\DB(k,n)$.

\begin{theorem}[BEST theorem~\cite{vanAardenne-Ehrenfest1987}]
 The number of distinct Eulerian cycles in an Eulerian directed graph $G=(V,E)$ is given by
\begin{equation}\label{BEST}
  e(G)=\kappa(G)\prod_{j=1}^n(d_{v_j}-1)!
\end{equation}
 where $d_{v_j}$ is the outdegree of $v_j$. %
\end{theorem}

 By definition, the outdegree of every node in the generalized de Bruijn graph $\DB(k,n)$ is $k$, therefore 
$\prod_{j=1}^n(d_{v_j}-1)!=((k-1)!)^n$. So, we have the following result.

\begin{theorem}\label{thm:counting}
For every $n\geq 1$, the number of generalized de Bruijn words of length $kn$ over $\Sigma_k$ is 
\begin{equation}\label{eqDBW}
    \DBW_k(kn)=((k-1)!)^n \kappa(\DB(k,n)).
\end{equation}
\end{theorem}

\begin{table}[ht]
\begin{center}
\begin{tabular}{cccccccccccccccccccc}
$n$    & 2 & 3 & 4 & 5 & 6  & 7  & 8  & 9  & 10 & 11   \\ \hline
$\DBW_2(2n)$  & 1& 1& 2& 3& 4& 7& 16& 21& 48& 93\\ \hline
$\DBW_3(3n)$  & 4& 24& 64& 512& 1728& 13312& 32768& 373248& 1310720& 10903552
\end{tabular}
\end{center}
\caption{The first few values of $\DBW_2(2n)$  and $\DBW_3(3n)$ (resp.~sequence A027362 and A192513 in~\cite{oeis}). \label{tab:DBW}}
\end{table}

The number of spanning trees can be computed by means of a determinant, as an application of the matrix-tree theorem below (see~\cite{DBLP:journals/jct/Levine11}). %
Recall that the \emph{Laplacian matrix} of a directed graph $G$ is the $n\times n$ matrix $L_G=D_G-A_G$, where $D_G$ and $A_G$ are the degree matrix and the adjacency matrix of $G$, respectively. It can be defined by
$$(L_G)_{ij}=\begin{cases}
                    d_{v_i} - d_{v_iv_i}  & \mbox{ if }   i = j \\
                    - d_{v_iv_j} & \mbox{ if }   i \neq j
                 \end{cases}$$%
where $d_{v_iv_j}$ is the number of edges directed from $v_i$ to $v_j$, and $d_{v_i} = \sum_j d_{v_iv_j}$ is the outdegree of $v_i$, so that  $d_{v_i} - d_{v_iv_i}$ is the outdegree of $v_i$ minus the number of its self loops.

\begin{theorem}[Matrix-Tree Theorem  for Eulerian graphs]\label{K}
 Let $G$ be an Eulerian graph. The number $\kappa(G)$ is equal to the determinant of the matrix obtained from the Laplacian matrix of $G$ after removing the last row and the last column. 
\end{theorem}

The matrix $L_G^0$ obtained by removing the last row and column from $L_G$ is sometimes called the \textit{reduced Laplacian matrix} of $G$. 

We now discuss the relations between generalized de Bruijn words and \emph{sandpile groups}, algebraic objects capturing important properties of directed Eulerian graphs, which can be defined in terms of the reduced Laplacian matrix of the graph. We start by giving some standard results and definitions from~\cite{DBLP:journals/jct/Stanley16}.
Recall that any integer matrix $A$ can be written in a canonical form $A=PDQ$, where both matrices $P$, $Q$ are unimodular integer matrices and $D$ is an integer diagonal matrix with the property that if $d_1,\ldots, d_n$ are the nonzero entries of the main diagonal of $D$, then $d_i | d_{i+1}$ for every $i$. The matrix $D$ is called the \emph{Smith Normal Form} of $A$. The $d_i$ are called the \emph{invariant factors} of the matrix $A$.

Let $G=(V, E)$ be a finite strongly connected directed graph, that is, for any $v, w \in V$ there are directed paths in $G$ from $v$ to $w$ and from $w$ to $v$. Then associated to any vertex $v$ of $G$ is an abelian group $K(G, v)$, the \textit{sandpile group} (also known as critical
group, Picard group, Jacobian, and group of components; see~\cite{DBLP:journals/jct/Stanley16} for the detailed definition). %
When $G$ is Eulerian, the groups $K(G, v)$ and $K(G, u)$ are isomorphic for any $v,u\in V$, and we let $K(G)$ denote the \emph{sandpile group of $G$}, and one has:
\begin{equation}\label{isomorphism}
    K(G)\cong \bigoplus_{i=1}^{n-1} \Z_{d_i}
\end{equation}
where $d_1,\ldots, d_{n-1}$ are the invariant factors of $L_G^0$, or equivalently, $d_1,\ldots, d_{n-1},0$ are  the invariant factors of $L_G$.

By the matrix-tree theorem, the order of the sandpile group %
$|K(G)|=\kappa(G)$ %
is equal to  $\det(L_G^0)$, the determinant of the reduced Laplacian matrix of $G$.

\begin{example}
Consider the generalized de Bruijn graph $\DB(3,6)$ (Fig.~\ref{fig:DB36}). %
Its Laplacian matrix is 
\[L_{\DB(3,6)}=\begin{pmatrix}
2 & -1 & -1 & 0 & 0 & 0 \\
0 & 3 & 0 & -1 & -1 & -1 \\
-1 & -1 & 2 & 0 & 0 & 0 \\
0 & 0 & 0 & 2 & -1 & -1 \\
-1 & -1 & -1 & 0 & 3 & 0 \\
0 & 0 & 0 & -1 & -1 & 2
\end{pmatrix}\]
and the Smith Normal Form of the associated reduced Laplacian matrix is
    \[\SNF(L^0_{\DB(3,6)})=\begin{pmatrix}
 1& 0& 0& 0& 0 \\
 0& 1& 0& 0& 0 \\
 0& 0& 3& 0& 0 \\
 0& 0& 0& 3& 0 \\
 0& 0& 0& 0& 3 
\end{pmatrix}\]

  Hence, by \eqref{isomorphism}, we have $K(\DB(3,6))\cong \Z_3^3$. The matrix $\SNF(L^0_{\DB(3,6)})$ has determinant $3^3=27$, so $\kappa(\DB(3,6))=27$. %
  Thus, by Theorem~\ref{thm:counting}, there are $27 \cdot 2^6 = 1728$ ternary generalized de Bruijn words of length $18$—the same as the number of distinct Eulerian cycles in $\DB(3,6)$ and Hamiltonian cycles in $\DB(3,18)$.

As another example, the Laplacian matrix of $\DB(2,6)$ is 
\[L_{\DB(2,6)}=
\begin{pmatrix}
1 & -1 & 0 & 0 & 0 & 0 \\
0 & 2 & -1 & -1 & 0 & 0 \\
0 & 0 & 2 & 0 & -1 & -1 \\
-1 & -1 & 0 & 2 & 0 & 0 \\
0 & 0 & -1 & -1 & 2 & 0 \\
0 & 0 & 0 & 0 & -1 & 1
\end{pmatrix}\]
and the Smith Normal Form of the associated reduced Laplacian matrix is 
    \[\SNF(L^0_{\DB(2,6)})=\begin{pmatrix}
 1& 0& 0& 0& 0 \\
 0& 1& 0& 0& 0 \\
 0& 0& 1& 0& 0 \\
 0& 0& 0& 2& 0 \\
 0& 0& 0& 0& 2 
\end{pmatrix}\]
We have $K(\DB(2,6))\cong \Z_2^2$, and in fact there are $4\cdot  (1!)^6=4$ binary generalized de Bruijn words of length $12$, namely $[000010111101]$, $[000011101101]$, $[000100101111]$, and $[000100111011]$.
 \end{example}
 
For ordinary de Bruijn graphs over a binary alphabet, the general structure of the sandpile group was %
described by Levine~\cite{DBLP:journals/jct/Levine11}:
 \begin{equation}\label{Levine}
     K(\DB(2,2^d))\cong \bigoplus_{i=1}^{d-1}(\Z_{2^i})^{2^{d-1-i}}.
 \end{equation}

In the case of generalized de Bruijn graphs, Chan, Hollmann and Pasechnik gave the structure of $K(\DB(k, n))$ for arbitrary $k$ and $n$~\cite[Theorem 3.1]{ECCGTA13}. We refer the reader to~\cite{ECCGTA13,CHAN2015268} for further details.

In the next section, we study the factor complexity of generalized de
Bruijn words and the effect of the BWT on it.

%%%%%%%%%%%%%%%%%%%%%%%%%%%%%%%%%%%%%
\section{Factor complexity of generalized de Bruijn words}\label{sec:factors}
%%%%%%%%%%%%%%%%%%%%%%%%%%%%%%%%%%%%%

For a word $w$, let $p_w(\ell)$ denote the number of its distinct ordinary
(noncyclic) factors of length $\ell$.  The map
$\ell\mapsto p_w(\ell)$ is the \emph{factor-complexity function} of $w$.
Following~\cite{KociumakaNavarroPrezza2023} and references therein, we define
$$
    \delta(w)=\max_{1\leq \ell\leq |w|}
    \frac{p_w(\ell)}{\ell}.
$$
For a necklace $[w]$, let $p_{[w]}(\ell)$ denote the number of its distinct
cyclic factors of length $\ell$.  We call $\ell\mapsto p_{[w]}(\ell)$ its
\emph{cyclic factor-complexity function}, and define 
$$
    \delta([w])=\max_{1\leq \ell\leq |w|}
    \frac{p_{[w]}(\ell)}{\ell}.
$$
Equivalently, $p_{[w]}(\ell)=p_{ww}(\ell)$ for any representative $w$ and
$1\leq\ell\leq|w|$, so this definition is independent of the representative.
Moreover, at most $\ell-1$ cyclic factors of length $\ell$ cross the boundary
of a fixed representative.  Hence
$\delta(w)\leq\delta([w])\leq\delta(w)+1$.

The measure $\delta$ is a universal lower bound on several measures associated
with dictionary compression, and every word $w$ admits a representation of
size $O(\delta(w)\log(|w|/\delta(w)))$ supporting efficient access and pattern
matching~\cite{KociumakaNavarroPrezza2023}.  It therefore provides a precise
way to ask whether generalized de Bruijn words retain the factor richness
expected from a generalization of de Bruijn words.  For ordinary de Bruijn
words, an asymptotically maximal value is immediate: a necklace $[w]$ of order
$d$ has all $k^d$ words of length $d$ as cyclic factors, and hence
$\delta([w])=\Theta(|w|/\log_k|w|)$.  A generalized de Bruijn word need not
contain all words of any fixed length.  Nevertheless, our first result shows
that every such word retains the same asymptotically maximal value over a fixed
alphabet.

\begin{theorem}\label{thm:delta-gdbw}
Let $[w]$ be a generalized de Bruijn word of length $kn$ over
$\Sigma_k$. Then
$$
    \delta([w])=\Theta\!\left(\frac{kn}{\log_k(kn)}\right).
$$
In particular, for fixed $k\geq2$, this simplifies to
$\delta([w])=\Theta(n/\log n)$.
\end{theorem}

\begin{proof}
The case $n=1$ is immediate, so assume that $n\geq 2$.
For a word $x$, let $\operatorname{occ}_{[w]}(x)$ denote its set of cyclic
occurrences in $[w]$.  For $\ell\geq 1$, let
$M_\ell=\max_{x\in\Sigma_k^\ell}|\operatorname{occ}_{[w]}(x)|$.
We first establish the only occurrence bound that we need.  Fix a word
$y$ and two letters $a,b\in\Sigma_k$.  The rows of the BWT matrix that
begin with $y$ form an interval, and the number of letters $a$ in the
corresponding interval of the last column is
$|\operatorname{occ}_{[w]}(ay)|$.  Since $\BWT([w])$ is a concatenation
of alphabet-permutations, the complete blocks in this interval contain
the same number of $a$'s and $b$'s; only its two boundary blocks can be
incomplete.  Therefore
$$
 \bigl||\operatorname{occ}_{[w]}(ay)|
       -|\operatorname{occ}_{[w]}(by)|\bigr|\leq 2.
$$
Moreover, every cyclic occurrence of $y$ has one preceding letter, so
$\sum_{b\in\Sigma_k}|\operatorname{occ}_{[w]}(by)|
=|\operatorname{occ}_{[w]}(y)|$.
It follows, for every $a\in\Sigma_k$, that
$$
    k|\operatorname{occ}_{[w]}(ay)|-2(k-1)
    \leq |\operatorname{occ}_{[w]}(y)|,
$$
and consequently
$$
    M_{\ell+1}\leq\frac{M_\ell+2(k-1)}{k}.
$$
Each letter occurs $n$ times in $[w]$, hence $M_1=n$.  Induction gives
$$
    M_\ell\leq
    2+\frac{n-2}{k^{\ell-1}}.
$$
Set $h=\left\lceil\log_k(n-1)\right\rceil+1$.
Since $k^{h-1}\geq n-1$, we obtain
$$
    M_h\leq
    \left\lfloor 2+\frac{n-2}{k^{h-1}}\right\rfloor\leq 2.
$$
There are $kn$ cyclic starting positions in $[w]$, and every length-$h$
cyclic factor occurs at most twice.  Therefore
$p_{[w]}(h)\geq \left\lceil kn/2\right\rceil$.
Since
$h\leq2\log_k(kn)\leq2\log_2(kn)=o(kn)$, this gives
$$
    \delta([w])\geq\frac{p_{[w]}(h)}{h}
    =\Omega\!\left(\frac{kn}{\log_k(kn)}\right).
$$

Conversely, every $k$-ary necklace of length $kn$ satisfies
$p_{[w]}(\ell)\leq\min\{k^\ell,kn\}$.
If $\ell\leq\frac12\log_k(kn)$, then
$p_{[w]}(\ell)/\ell\leq\sqrt{kn}$; otherwise,
$p_{[w]}(\ell)/\ell<2kn/\log_k(kn)$.  This proves the matching upper bound.
\end{proof}

Theorem~\ref{thm:delta-gdbw} supplies a class of maximally factor-rich inputs
for studying the compression effect of the BWT.  Bannai, I, and
Nakashima~\cite{BannaiINakashima2025} showed that the Bijective BWT (BBWT), a
bijective variant of the BWT, can map a family with asymptotically maximal
$\delta$ to one with constant $\delta$.  This naturally raises the same
question for the ordinary BWT.  We answer it affirmatively by constructing
generalized de Bruijn words whose BWTs have $\delta=2$.  In view of the
dictionary-compression interpretation of $\delta$ recalled above, it shows
that the ordinary BWT can substantially improve dictionary compressibility.

Recall from Section~\ref{sec:prel} the binary Thue--Morse morphism $\tau$,
defined by $\tau(0)=01$ and $\tau(1)=10$.

For coprime positive integers $a$ and $m$, let $\ord_m(a)$ denote the
multiplicative order of $a$ modulo $m$.  Such an integer $a$ is called a
\emph{primitive root modulo $m$} if $\ord_m(a)=\phi(m)$.

\begin{lemma}\label{lem:order-four-three-power}
For every $r\geq1$,
$$
    \ord_{3^r}(4)=3^{r-1}.
$$
\end{lemma}

\begin{proof}
The standard lifting theorem for primitive roots states that, for an odd
prime $p$, every primitive root modulo $p^2$ is a primitive root modulo
$p^r$ for all $r\geq2$~\cite[Theorem~2.40, p.~102]{NivenZuckermanMontgomery1991}.
Now $2$ has order $6$ modulo $9$, since
$2^2\not\equiv1\pmod9$, $2^3\equiv-1\pmod9$, and
$2^6\equiv1\pmod9$.  Thus $2$ is a primitive root modulo $9$; it is
also a primitive root modulo $3$ because $2\equiv-1\pmod3$.  The cited
theorem therefore gives
$\ord_{3^r}(2)=\phi(3^r)=2\cdot3^{r-1}$ for every $r\geq1$.
In general, if an element $g$ has finite order $h$, then $g^2$ has order
$h/\gcd(h,2)$: indeed, this is the least positive integer $s$ such that
$h$ divides $2s$.  Applying this identity to $g=2$ gives
$$
    \ord_{3^r}(4)
    =\frac{\ord_{3^r}(2)}{\gcd(\ord_{3^r}(2),2)}
    =3^{r-1}.
$$
\end{proof}

\begin{theorem}\label{thm:delta-two-bwt}
For every $d\geq0$, set $n_d=3^d$.  Then
$$
    L_d=\tau\!\left((10)^{(n_d-1)/2}1\right)
    =(1001)^{(n_d-1)/2}10
$$
is the BWT of a binary generalized de Bruijn word
$[u_d]$ of length $2n_d$, and
$$
    \delta(L_d)=2.
$$
\end{theorem}

\begin{proof}
Let $q_d=\pi_{L_d}^{-1}$ be the inverse standard permutation of
$L_d$.  Its $j$th block is $10$ when $j$ is even and $01$ when $j$ is
odd.  Since each block contains exactly one occurrence of each letter, the
$j$th $0$ and the $j$th $1$ both occur in the $j$th block.  Hence
$$
    q_d(j)=
    \begin{cases}
        2j+1,&\text{if $j$ is even},\\
        2j,&\text{if $j$ is odd},
    \end{cases}
    \qquad
    q_d(n_d+j)=
    \begin{cases}
        2j,&\text{if $j$ is even},\\
        2j+1,&\text{if $j$ is odd},
    \end{cases}
    \qquad(0\leq j<n_d).
$$
Since $n_d$ is odd, the map $(a,p)\mapsto 2a+p$ is a bijection from
$\Z/n_d\Z\times\{0,1\}$ to $\Z/2n_d\Z$.  We use
$(a,p)$ as coordinates for the corresponding residue.  Let $x$ be represented
by $2a+p$, so that $p$ is the parity of $x$.  If $x=j<n_d$, then $j$ has
parity $p$, and the first formula gives $q_d(x)=2x+1-p$.
If $x=n_d+j$, then $j$ has parity $1-p$, because $n_d$ is odd.  Thus
the second formula gives
$q_d(x)=2j+1-p\equiv2x+1-p\pmod{2n_d}$.
In either case, substituting $x=2a+p$ shows that $q_d(x)$ is represented by
$4a+p+1$.  This residue has parity $1-p$, and its first coordinate is
$\bigl(4a+p+1-(1-p)\bigr)/2=2a+p$ in $\Z/n_d\Z$.  Consequently,
$$
    q_d(a,p)=(2a+p,1-p),
    \qquad
    q_d^2(a,p)=(4a+p+1,p).
$$
It follows by induction that, for every $t\geq1$,
$$
    q_d^{2t}(0,0)
    =\left(1+4+\cdots+4^{t-1},0\right)
    =\left(\frac{4^t-1}{3},0\right),
$$
where the first coordinate is an element of $\Z/n_d\Z$.  Consequently,
for $t\geq1$,
$$
    q_d^{2t}(0,0)=(0,0)
    \quad\Longleftrightarrow\quad
    n_d\mid 1+4+\cdots+4^{t-1}
    \quad\Longleftrightarrow\quad
    3^{d+1}\mid4^t-1.
$$
By Lemma~\ref{lem:order-four-three-power},
$\ord_{3^{d+1}}(4)=3^d=n_d$.
Thus the first positive $t$ such that $q_d^{2t}(0,0)=(0,0)$ is $t=n_d$.
Furthermore, $q_d^{2t+1}(0,0)$ has second coordinate $1$ for all $t\geq0$,
so it cannot equal $(0,0)$.
Hence the orbit of $0$ under $q_d$ has length $2n_d$, and $q_d$ is a single
cycle.

The inverse of a cycle is a cycle, so $\pi_{L_d}$ is also a
cycle.  By the BWT-image criterion recalled in Section~\ref{sec:prel}, there
is an aperiodic necklace $[u_d]$ such that
$L_d=\BWT([u_d])$.  Moreover,
$L_d=\tau((10)^{(n_d-1)/2}1)$ is an alphabet-permutation power, so
$[u_d]$ is a generalized de Bruijn word.

Finally, $L_d$ is a prefix of the period-$4$ word
$(1001)^\omega$.  Hence $p_{L_d}(\ell)\leq4$ for
$2\leq\ell\leq2n_d$.  Both letters occur in $L_d$, so
$p_{L_d}(1)=2$.  Therefore $\delta(L_d)=2$.
\end{proof}

Combining Theorems~\ref{thm:delta-gdbw} and~\ref{thm:delta-two-bwt}, we
obtain, as $d$ tends to infinity,
$$
    \delta([u_d])=\Theta\!\left(\frac{3^d}{d}\right),
    \qquad\text{whereas}\qquad
    \delta(\BWT([u_d]))=2.
$$
Since $\delta(u_d)\geq\delta([u_d])-1$ for every representative $u_d$, the
input has asymptotically maximal ordinary factor complexity.  Thus
$\delta$ provides an asymptotically maximal lower bound on the associated
dictionary-compression measures before the BWT.  By contrast, the
representation bound recalled above has size $O(d)$ for $\BWT([u_d])$.  Hence
the ordinary BWT can indeed substantially improve dictionary compressibility.

We now turn to alphabets of prime size $p$.  This assumption allows us to
relate generalized de Bruijn words to necklaces whose BWT matrices are
invertible over $\Z_p$.

\section{Invertible necklaces and Reutenauer groups}\label{sec:invertible}

In this section, we define invertible necklaces and highlight their connections with abstract algebra.

In \cite{AAM2026_103081} (see also the related paper \cite{ReSh26}), the authors used BWT matrix multiplication to obtain a group structure on important classes of binary words (namely, \emph{Christoffel words} and \emph{balanced} words) that have an invertible BWT matrix. Following this idea, we define \emph{invertible} necklaces: %

\begin{definition}
    Let $p$ be a prime. A necklace $[w]$ over the alphabet $\Sigma_p=\{0,1,\ldots,p-1\}$ is called \emph{invertible} if its BWT matrix is nonsingular, i.e., has nonzero determinant modulo $p$.
\end{definition}

As observed in \cite{AAM2026_103081}, the product of two BWT matrices is not, in general, a BWT matrix; but the author suggests replacing BWT matrices by \emph{circulant matrices}. As we will see, using circulant matrices, we can define the product between any two invertible necklaces.

Let $w$ be a word over $\Sigma_k$. The \textit{circulant matrix} of $w$ is the square matrix $\CM_w$ whose $(i+1)$th row is $\sigma^{i}(w)$. For example, the circulant matrix of $011$ is $\CM_{011}=\begin{pmatrix}
0\, 1\, 1 \\
1\, 0\, 1 \\
1\, 1\, 0
\end{pmatrix}$.

One can immediately observe that, since the circulant matrix of a word can be obtained by permuting the rows of its BWT matrix, a necklace is invertible if and only if any (and then, every) of its associated circulant matrices is invertible.

Recall that for every prime $p$ and any positive integer $n$, there is a unique (up to isomorphism) finite field with $p^n$ elements, denoted $\GF_{p^n}$, and its multiplicative group $\GF_{p^n}^*$ is cyclic. The set of invertible $n \times n$-circulant matrices over $\GF_p$ forms, with respect to matrix multiplication, an abelian group $C(p, n)$. We consider the group $C(p,n)/\langle Q_n \rangle$, where $Q_n$ is the permutation matrix of the cycle  $(0\, 1\, \ldots \, n-1)$. This group is called the \emph{Reutenauer group} $\RG_p^n$ in \cite{DuzhinJMS}. Intuitively, an element of the Reutenauer group corresponds to an equivalence class of invertible circulant matrices, where two circulant matrices are equivalent if and only if they are the circulant matrices of two conjugated words. Thus, $\RG_p^n$ is in natural bijection with the set of invertible necklaces of length $n$ over $\Sigma_p$, and one can write $\CM_{[w]}=\{\CM_v,~v\in[w]\}\in \RG_p^n$ for the class of circulant matrices associated with a given invertible necklace $[w]$ of length $n$ over $\Sigma_p$. In particular, this induces a group structure on the set of such invertible necklaces, isomorphic to $\RG_p^n$, namely with respect to the multiplication defined by $[u]\cdot [v]=[w]$ where $[w]$ is such that $\CM_{[u]}\cdot \CM_{[v]}=\CM_{[w]}$. This multiplication does not depend on the choice of representatives for each necklace (equivalently, for each matrix).

As shown in~\cite{DuzhinJMS}, the Reutenauer group acts on the set of all aperiodic necklaces %
as follows:
Let $[v]$ be an aperiodic necklace of length $n$ over $\Sigma_p$.  For any $\CM_{[w]}\in \RG_p^n$, we define $\CM_{[w]}\cdot [v]$ as %
the necklace of $\CM_{w}\cdot v^T$. %
In particular, this operation does not depend on the choice of a representative.
For example, $\CM_{[11010]}\cdot [11000]=[11110]$ since $\begin{pmatrix}
1\, 1\, 0\, 1\, 0 \\
0\, 1\, 1\, 0\, 1 \\
1\, 0\, 1\, 1\, 0\\
0\, 1\, 0\, 1\, 1\\
1\, 0\, 1\, 0\, 1
\end{pmatrix}\cdot
\begin{pmatrix}
1 \\
1 \\
0\\
0\\
0
\end{pmatrix}= 01111
$.

The orbit of an aperiodic necklace $[v]$ is therefore $O_{[v]}=\{[\CM_{[w]}\cdot [v]] \mid \CM_{[w]}\in \RG_p^n\}$. 

We now recall some classical results and definitions in abstract algebra (see, for example,~\cite{Lidl_Niederreiter_1996} for a more detailed presentation), which have a strong connection with circulant matrices and invertible necklaces.

It is well known that the number of irreducible polynomials over $\GF_p$ of degree $n$ is equal to the number of aperiodic necklaces of length $n$ over $\Sigma_p$, but there is no known canonical bijection between the two sets. 
The \emph{minimal polynomial} of a nonzero element $\alpha$ of $\GF_{p^n}$ is the (unique) monic polynomial $f\in\GF_p[X]$ with the least degree for which $\alpha$ is a root. This polynomial has as roots all the $m$ distinct elements of the form $\alpha, \alpha^{p},\ldots, \alpha^{p^{m-1}}$, where $m$ is the least integer such that $\alpha^{p^{m}}=\alpha$, i.e., the distinct roots of $f$ are the orbit of $\alpha$ under the application of the Frobenius automorphism $\alpha\mapsto \alpha^p$. Thus, $f$ has degree $m$ and can be factored as
\[f=(x-\alpha)(x-\alpha^{p})\cdots (x-\alpha^{p^{m-1}}).\]
The sum of all roots of $f$ is called the \emph{trace} of $f$. %

But $\GF_{p^n}$ is also a vector space over $\GF_{p}$. An element $\gamma$ of $\GF_{p^n}$ is called \emph{normal} if  $\{\gamma, \gamma^{p},\ldots,\gamma^{p^{n-1}}\}$ forms a vector space basis for $\GF_{p^n}$ over $\GF_{p}$, called a \emph{normal basis}. 

The minimal polynomial of a normal element of $\GF_{p^n}$ over $\GF_{p}$ is called a \emph{normal polynomial}. %
Normal polynomials have nonzero trace (modulo $p$).

If $\gamma$ is a normal element of $\GF_{p^n}$, every element of $\GF_{p^n}$ can be written as a linear combination of elements in the normal basis $\{\gamma, \gamma^{p},\ldots,\gamma^{p^{n-1}}\}$. That is, we can write all the elements of $\GF_{p^n}$ as $n$-length words, or  $n$-length  vectors. %
More precisely, if $\alpha=w_0\gamma+w_1\gamma^p+\ldots +w_{n-1}\gamma^{p^{n-1}}$, $w_j\in\GF_p$ for every $j=0,\ldots,n-1$, we can represent $\alpha$ with the word $w=w_0w_1\cdots w_{n-1}$. 
The application of the Frobenius automorphism to an element of the field corresponds to the shift of the corresponding word, and an orbit to a conjugacy class, i.e., to a necklace.  This necklace is aperiodic if and only if the orbit is maximal, i.e., has size $n$.

\begin{table}[ht]
  \begin{center}
    \caption{The finite field $\GF_{2^4}$ as a vector space over $\GF_2$, where the elements are grouped by orbits with respect to a normal element $\gamma$.}
    \label{tableGFV}
    \begin{tabular}{l|l|l|l} 
     orbit & orbit as vectors (words)  & normal \\
      \hline
    $\{0\}$ & $\{0000\}$ & no \\
    $\{\gamma,\gamma^2,\gamma^4,\gamma^8\}$ & $\{1000,0100,0010,0001\}$  & yes\\
    $\{\gamma+\gamma^2,\gamma^2+\gamma^4,\gamma^4+\gamma^8,\gamma+\gamma^8\}$ & $\{1100,0110,0011,1001\}$  & no\\
    $\{\gamma^2+\gamma^8,\gamma+\gamma^4\}$ & $\{0101,1010\}$  & no\\
    $\{\gamma+\gamma^2+\gamma^4,\gamma^2+\gamma^4+\gamma^8,\gamma+\gamma^4+\gamma^8,\gamma+\gamma^2+\gamma^8\}$ & $\{1110,0111,1011,1101\}$ &  yes\\
    $\{\gamma+\gamma^2+\gamma^4+\gamma^8\}$  & $\{1111\}$ & no 
    \end{tabular}
  \end{center}
\end{table}

Fixing a normal element $\gamma$ and writing elements of $\GF_{p^n}$ in the induced normal basis, the orbit of $\gamma$ corresponds to the necklace $[10^{n-1}]$. There may be other orbits constituted by normal elements. For example, the orbit corresponding to the necklace $[1110]$ is another orbit of normal elements in $\GF_{2^4}$.

The number of normal elements of $\GF_{p^n}$ is %
counted by $\Phi_p(n)$, the generalized Euler's totient function. Equivalently, $\Phi_p(n)$ %
counts the number of polynomials over $\GF_p$ of degree smaller than $n$ and coprime with $X^n-1$. It can be computed using the formula
\[\Phi_p(n)=p^n\prod_{d | (n/\lambda_p(n))}\left(1-\dfrac{1}{p^{\ord_d(p)}}\right)^{\tfrac{\phi(d)}{\ord_d(p)}}\]
for $n>1$, 
where $\lambda_p(n)$ is the largest power of $p$ dividing $n$, and $\ord_d(p)$ is the multiplicative order of $p$ modulo $d$. The first mention of this formula seems to come from~\cite{Hensel1888230237}, and a more recent study can be found in~\cite{hyde2018normalelementsfinitefields}.%

The first few values of the sequences $\Phi_2(n)$ and $\Phi_3(n)$ are presented in Table~\ref{tab:Phi}.

\begin{table}[ht]
\begin{center}
\begin{tabular}{cccccccccccccccccccc}
$n$    & 1 & 2 & 3 & 4 & 5 & 6  & 7  & 8  & 9  & 10 & 11  & 12  \\ \hline
$\Phi_2(n)$  & 1& 2& 3& 8& 15& 24& 49& 128& 189& 480& 1023& 1536\\ \hline
$\Phi_3(n)$  & 1& 4& 18& 32& 160& 324& 1456& 2048& 13122& 25600& 117128 &209952
\end{tabular}
\end{center}
\caption{First few values of $\Phi_2(n)$  and $\Phi_3(n)$ (resp.~sequence A003473 and A003474 in~\cite{oeis}).\label{tab:Phi}}
\end{table}

Let $p$ be a prime, and consider words in $\Sigma_p^n$ as vectors in the vector space $\GF_{p^n}$. %

An $n \times n$-circulant matrix over $\GF_p$ is in $C(p,n)$ if and only if it is invertible (or nonsingular), if and only if its determinant is nonzero modulo $p$; or, equivalently, if 
its rows form a normal basis of $\GF_{p^n}$.

Since every element of a normal basis can be chosen as the first row of a corresponding invertible circulant matrix, the order of $C(p,n)$ is $\Phi_p(n)$, and the order of $\RG_p^n$ (hence, the number of invertible necklaces) is $\Phi_p(n)/n$.

We therefore arrive at the following characterization:

\begin{theorem}\label{thm:invneckeq}
     Let $[w]$ be a necklace  over $\Sigma_p$, $p$ prime. The following are equivalent:
    \begin{enumerate}
        \item The necklace $[w]$ is invertible;
        \item Any word in $[w]$ has an invertible circulant matrix over $\GF_p$;
        \item $\CM_{[w]}$ belongs to the Reutenauer group $\RG_p^n$;
        \item Any word in $[w]$ is a vector corresponding to a normal element of $\GF_{p^n}$;
        \item Every word of length $n$  over $\Sigma_p$ can be written in a unique way as a unique $\mathbb F_p$-linear combination of the $n$ distinct rotations of $w$.
        %sum (modulo $p$) of words in $[w]$.
    \end{enumerate}
\end{theorem}

\begin{example}
Let $p=2$ and $n=4$. We have three aperiodic binary necklaces, namely $[1000]$, $[1100]$, and $[1110]$. Out of them, the first and the last are invertible, while the second is not, since $\det \begin{pmatrix}
1\, 1\, 0\, 0\\
0\, 1\, 1\, 0 \\
0\, 0\, 1\, 1\\
1\, 0\, 0\, 1
\end{pmatrix}=0$. Indeed, $1100+0110+0011=1001$ in $\GF_{2^4}$, so the elements are not linearly independent as vectors, and therefore do not form a basis of $\GF_{2^4}$.

The Reutenauer group $\RG_2^4$ is therefore constituted by the classes of matrices $Id=\CM_{[1000]}$
and
$\CM_{[1110]}$.
\end{example}

So, the set of aperiodic necklaces of length $n$ over $\Sigma_p$ with invertible BWT matrix forms an abelian group isomorphic to the Reutenauer group $\RG_p^n$, but the multiplication has to be carried out with the circulant matrices rather than with BWT matrices. This was also speculated in a remark we recently found at the end of \cite{AAM2026_103081} (see also the related paper \cite{ReSh26}).

\begin{lemma}\label{lem:trace}
Let $[w]$ be an aperiodic necklace over $\Sigma_p$ and let $f$ be the minimal polynomial of $w$, seen as a vector. Then, $f$ has trace $(m,\ldots,m)$, where $m=wt([w])\mod p$. If $[w]$ is invertible, then $m \neq 0$.
\end{lemma}
\begin{proof}
Since the roots of $f$ are, as vectors, shifts of $w$, the trace of $f$ is the vector $(m,\ldots, m)$. Furthermore, if $[w]$ is invertible, the rows of $\CM_w$ are linearly independent. Since their sum is $m(1,\ldots,1)$, one must have $m\neq 0$.
\end{proof}

In particular, binary invertible necklaces have an odd number of $1$'s. However, note that, in general, having a nonzero trace is not a sufficient condition for an element to be normal.

Recall that a prime $q$ is called \emph{$p$-rooted} if $p$ is a primitive root modulo $q$. For example, $5$ is $2$-rooted since
$\{2^i\mod 5 \mid i=1,\ldots,4\}=\GF_{5}^*$, 
while $7$ is not $2$-rooted since 
$\{2^i\mod 7\mid i=1,\ldots,6\}=\{1,2,4\}$.

%In \cite{CTR01}, the authors obtained the following result:

\begin{theorem}[\cite{CTR01}]\label{thm:remarkable}
    Let $n$ be a positive integer. Every monic irreducible polynomial of degree $n$ over $\GF_p$ with nonzero trace is normal if and only if $n$ is either a power of $p$ or a $p$-rooted prime.
\end{theorem}

The latter result, together with %
Lemma~\ref{lem:trace}, immediately implies the following:

\begin{corollary}\label{cor:ArtinBWT}
     Let $n$ be a positive integer. Every aperiodic necklace of length $n$ over $\Sigma_p$ with nonzero weight modulo $p$ is invertible if and only if $n$ is either a power of $p$ or a $p$-rooted prime.
\end{corollary}

For example, the sequence of $2$-rooted primes is A001122 in OEIS~\cite{oeis}. Therefore, by Corollary~\ref{cor:ArtinBWT}, this sequence precisely corresponds to the lengths $n$ (excluding powers of $2$) for which all aperiodic binary necklaces of nonzero weight modulo $2$ are invertible. The latter is also the sequence of numbers of the form $2m+1$ such that $(10)^m$ is a BWT image (see \cite{Aulicino01061969,DBLP:journals/tcs/MantaciRRSV17}), i.e., lengths for which one can have a BWT with the largest possible number of runs (which could be considered as the \emph{worst case} of the BWT). Notice that words having BWT of the form $(10)^m$ are, by definition, generalized de Bruijn words.

However, we can state the following conjecture:

\begin{conjecture}\label{conj:ArtinBWT}
Let $p$ be a positive prime number. There exist infinitely many $n$, different from a power of $p$, for which every aperiodic necklace of length $n$ over $\Sigma_p$ with nonzero weight modulo $p$ is invertible.
\end{conjecture}

The previous conjecture is related to a famous conjecture of Emil Artin:

\begin{conjecture}\label{conj:Artin}
Let $a$ be an integer that is not a square number and not $-1$. The set of $a$-rooted primes is infinite.\footnote{The full statement of the conjecture also includes the density of the $a$-rooted primes under certain conditions. See Moree~\cite{Moree2012} for a detailed survey.}    
\end{conjecture}

Let $p$ be a positive prime number. From Corollary~\ref{cor:ArtinBWT}, the lengths $n$ that are not a power of $p$ such that every aperiodic necklace over $\Sigma_p$ with nonzero weight modulo $p$ is invertible are exactly the set of $p$-rooted primes. We therefore have the following:

\begin{corollary}
Conjecture~\ref{conj:ArtinBWT} is equivalent to the restriction of Conjecture~\ref{conj:Artin} to positive primes.
\end{corollary}

\section{Generalized de Bruijn words and invertible necklaces}\label{sec:final}

In this section, we further explore the connection between the content of Sections~\ref{sec:db} and~\ref{sec:invertible}. Specifically, we demonstrate that over an alphabet whose size is a prime number, generalized de Bruijn words and invertible necklaces are linked not only through their connection to the Burrows–Wheeler transform, but also via an isomorphism of abelian groups.

First, recall that every finite abelian group $G$ is isomorphic to the direct sum of cyclic groups, i.e., $G\cong \bigoplus_i \Z_{d_i}$. One has, for example, $\RG_2^4 \cong \mathbb{Z}_2$ and $\RG_2^8 \cong \mathbb{Z}_2^2 \oplus \mathbb{Z}_4$ (remember that the order of  $\RG_p^n$ is $\Phi_p(n)/n$).

 Chan, Hollmann and Pasechnik~\cite{ECCGTA13} proved that the structure of Reutenauer groups can be found by means of an isomorphism with the sandpile groups of generalized de Bruijn graphs, as reported in the next theorem.
 
 \begin{theorem}[\cite{ECCGTA13}]\label{thm:KDB-RG}
    Let $K(\DB(p,n))$ be the sandpile group of the generalized de Bruijn graph $\DB(p,n)$, $p$ prime. Then 
    \[ K(\DB(p,n)) \oplus \Z_{p-1} \cong \RG_p^{n},\] 
    where $\RG_p^n\cong C(p,n)/\langle Q_n \rangle$ is the $n$-th Reutenauer group over $\GF_p$.

    Thus, for $p$ prime,  $K(\DB(p,n))\cong  C(p, n)/(\Z_{p-1}\times \Z_n)$, and therefore  $\kappa(DB(p,n))=\dfrac{\Phi_p(n)}{n(p-1)}$.
 \end{theorem}

As a consequence of Theorem~\ref{thm:KDB-RG}, we arrive at a remarkable fact: the structure of the Reutenauer groups is entirely determined by the structure of the generalized de Bruijn graphs. An illustrative example is provided in Table~\ref{tabledi}. We point out that the (reduced) Laplacian matrix of a generalized de Bruijn graph (as well as its Smith Normal Form and invariant factors) can be derived directly from the arithmetic definition of the graph, without the need to construct it explicitly.

\begin{table}[ht]
  \begin{center}
    \caption{  The invariant factors of the reduced Laplacian matrix of the generalized de Bruijn graphs for $k=3$, which give the structure of the sandpile groups $K(\DB(3,n))$ and Reutenauer groups $\RG_3^n$.  \label{tabledi}}
    \begin{tabular}{l|l|l|l} %
     $n$ & $d_i$ & $K(\DB(3,n))$ & $\RG_3^n$\\
      \hline
    $3$ & $(1,3)$ & $\Z_3$  & $\Z_3\oplus \Z_2$\\
    $4$ & $(1,1,4)$ & $\Z_4$ & $\Z_4\oplus \Z_2$ \\
    $5$ & $(1, 1, 1, 16)$ & $\Z_{16}$ & $\Z_{16}\oplus \Z_2$\\
    $6$ & $(1, 1, 3, 3, 3)$ & $\Z_3^3$  & $\Z_3^3\oplus \Z_2$\\
    $7$ & $(1, 1, 1, 1, 1, 104)$  & $\Z_{104}$ & $\Z_{104}\oplus \Z_2$\\
    $8$ & $(1, 1, 1, 1, 2, 8, 8)$  & $\Z_{8}^2\oplus \Z_2$  & $\Z_{8}^2\oplus \Z_2^2$\\
    $9$ & $(1, 1, 1, 3, 3, 3, 3, 9)$ & $\Z_{9}\oplus \Z_3^4$ & $\Z_{9}\oplus \Z_3^4\oplus \Z_2$ \\
    $10$ & $(1, 1, 1, 1, 1, 1, 1, 16, 80)$ & $\Z_{80}\oplus \Z_{16}$ & $\Z_{80}\oplus \Z_{16}\oplus \Z_2$\\
    $11$ & $(1, 1, 1, 1, 1, 1, 1, 1, 22, 242)$ & $\Z_{242}\oplus \Z_{22}$ & $\Z_{242}\oplus \Z_{22}\oplus \Z_2$\\
    $12$ & $(1, 1, 1, 1, 3, 3, 3, 3, 3, 3, 12)$ & $\Z_{12}\oplus \Z_{3}^6$ & $\Z_{12}\oplus \Z_{3}^6\oplus \Z_2$
    \end{tabular}
  \end{center}
\end{table}
 
We thus obtain the following formula for the number of generalized de Bruijn words over an alphabet of prime size:

\begin{theorem}
Let $p$ be a prime. For every $n\geq 2$, the number of generalized de Bruijn words of length $pn$ over $\Sigma_p=\{0,1,\ldots,p-1\}$ is 
\begin{equation}\label{eq:DBWprime}
    \DBW_p(pn)=\dfrac{((p-1)!)^n}{p-1}\dfrac{\Phi_p(n)}{n}=\dfrac{((p-1)!)^n}{p-1}\InvNeck_p(n)
\end{equation}
where $\InvNeck_p(n)$ is the number of invertible  necklaces of length $n$ over $\Sigma_p$.
\end{theorem}

\begin{proof}
By Theorem~\ref{thm:KDB-RG}, we have $\kappa(\DB(p,n))=\Phi_p(n)/(n(p-1))$. The result then follows from \eqref{eqDBW}.
\end{proof}

\begin{lemma}
 When $n$ is a power of a prime $p$ (which corresponds to the case of ordinary de Bruijn words), the precise structure of the sandpile group (and hence of the Reutenauer group) can be easily given in terms of $p$ and $n$. Indeed, from the results in~\cite{ECCGTA13}, one can derive:

\begin{equation}\label{eq:SandpilePrime}
    K(\DB(p,p^n))\cong \big[ \bigoplus_{i=1}^{n-1}(\Z_{p^i})^{p^{n-1-i}(p^2-2p+1)}\big]\oplus \Z_{p^n}^{p-2},
\end{equation}
and hence, by Theorem~\ref{thm:KDB-RG},
\begin{equation}\label{eq:RG}
RG_p^n\cong \big[ \bigoplus_{i=1}^{n-1}(\Z_{p^i})^{p^{n-1-i}(p^2-2p+1)}\big]\oplus \Z_{p^n}^{p-2}\oplus \Z_{p-1}.
\end{equation}
Note that, as expected, for $p=2$ we obtain \eqref{Levine}.
 \end{lemma}
 \begin{proof}
We use the formula from~\cite{ECCGTA13} (Theorem 3.1), that describes the group $K(\DB(d,k))$ for arbitrary $d,k$, in the case where $d=p$ is prime and $k=p^n$ for $n\ge 1$:

\begin{align*}
 K(\DB(p,p^n))%&\cong \bigoplus_{i=0}^{n-1}\big[\Z_{p^{i}}\oplus \Z_{p^{i+1}}^{n_i-2n_{i+1}+n_{i+2}-1} \big]\\
 &\cong \big[ \bigoplus_{i=1}^{n-1}\Z_{p^i}\big]\oplus \big[ \bigoplus_{i=0}^{n-2}\Z_{p^{i+1}}^{p^{n-i}-2p^{n-i-1}+p^{n-i-2}-1} \big]\oplus \Z_{p^n}^{p^{n-(n-1)}-2p^{n-(n-1)-1}+p^{n-(n-1)-1}-1} \\             &\cong \big[ \bigoplus_{i=1}^{n-1}\Z_{p^i}\big]\oplus \big[ \bigoplus_{i=0}^{n-2}\Z_{p^{i+1}}^{p^{n-i-2}(p^2-2p+1)-1} \big]\oplus \Z_{p^n}^{p-2} \\
 &\cong \big[ \bigoplus_{i=1}^{n-1}\Z_{p^i}\big]\oplus \big[ \bigoplus_{i=1}^{n-1}\Z_{p^{i}}^{p^{n-i-1}(p^2-2p+1)-1} \big]\oplus \Z_{p^n}^{p-2} \\
 &\cong \big[ \bigoplus_{i=1}^{n-1}\Z_{p^{i}}^{p^{n-i-1}(p^2-2p+1)} \big]\oplus \Z_{p^n}^{p-2} 
 \end{align*}
\end{proof}

\begin{remark}
Let $pn=p^d$, $d> 1$, i.e., $n=p^{d-1}$. Since $\Phi_p(p^{d-1})=p^{p^{d-1}}\dfrac{p-1}{p}$, from \eqref{eq:DBWprime} we get 
\[\DBW_p(p^d)=\dfrac{((p-1)!)^{p^{d-1}}p^{p^{d-1}}}{pp^{d-1}}=\dfrac{(p!)^{p^{d-1}}}{p^d}\]
 i.e., the number of (ordinary) de Bruijn words of order $d$ over $\Sigma_p$.
\end{remark}

In the special case $p=2$, we have, by Theorem~\ref{thm:KDB-RG}, that $K(\DB(2,n))\cong \RG_2^n$, and in this case we have that for every $n\geq 1$, the following sets are equinumerous, hence admit a bijection: 
    \begin{itemize}
        \item The set of binary generalized de Bruijn words of length $2n$;
        \item The set of binary invertible necklaces of length $n$;
        \item The set of (unordered) normal bases of $\GF_{2^{n}}$.
    \end{itemize}

% \begin{theorem}\label{thm:final1}
%     For every $n\geq 1$, the following sets are equinumerous, hence admit a bijection: 
%     \begin{itemize}
%         \item The set of binary generalized de Bruijn words of length $2n$;
%         \item The set of binary invertible necklaces of length $n$;
%         \item The set of (unordered) normal bases of $\GF_{2^{n}}$.
%     \end{itemize}
% \end{theorem}

 In particular, when $n=2^d$ for some integer $d>0$, the number of binary de Bruijn words of length $2^{d+1}$ is equal to the number of binary necklaces of length $2^d$ having an odd number of $1$'s. This follows from Corollary~\ref{cor:ArtinBWT}. Indeed, notice that we do not need to add the hypothesis that the necklaces are aperiodic, since their lengths are a power of $2$. So, the following sets are equinumerous,  hence admit a bijection: 
    \begin{itemize}
        \item The set of binary de Bruijn words of order $d+1$;
        \item The set of binary necklaces of length $2^d$ having an odd number of $1$'s;
         \item The set of (unordered) normal bases of $\GF_{2^{2^d}}$.
    \end{itemize}

% \begin{theorem}\label{thm:final2}
%     For every $d\geq 1$, the following sets are equinumerous,  hence admit a bijection: 
%     \begin{itemize}
%         \item The set of binary de Bruijn words of order $d+1$;
%         \item The set of binary necklaces of length $2^d$ having an odd number of $1$'s;
%          \item The set of (unordered) normal bases of $\GF_{2^{2^d}}$.
%     \end{itemize}
% \end{theorem}

%\begin{remark}
% Theorem~\ref{thm:final2} is enumerative: its derivation does not provide an explicit bijection. 
Constructing such a bijection, and determining whether it can be evaluated efficiently, is therefore a natural problem. An efficient bijection would transfer uniform random generation of binary de Bruijn words~\cite{SawadaGabric2025} to the corresponding odd-weight necklaces and normal bases. Conversely, algorithms for constructing normal bases~\cite{GathenGiesbrecht1990} could yield algebraic constructions of de Bruijn words.
%\end{remark}

\section{Conclusions and Open Problems}\label{sec:conclusion}

As emphasized in the introduction, the central motivation of our work is to present a new formalism, based on combinatorics on words, that bridges various algebraic structures. This approach builds on known results relating sandpile groups of (generalized) de Bruijn graphs and groups of invertible matrices over finite fields (Reutenauer groups). The new classes of words we introduce in this paper---generalized de Bruijn words on the one hand, and conjugacy classes of words (necklaces) with invertible Burrows--Wheeler matrix on the other hand---offer a fresh reinterpretation of these connections in terms of words. The key advantage of the new framework is its foundation in the Burrows--Wheeler transform, which not only clarifies links to longstanding open problems in elementary number theory but also opens new avenues for further research.

Our first result on $\delta$ confirms that generalized de Bruijn
words retain the factor richness expected from a generalization of de Bruijn
words.  They have
the same asymptotically maximal factor complexity as ordinary de
Bruijn words over a fixed alphabet.  We then use this factor-rich class to
obtain a more general result about the compression effect of the BWT.  For the
infinite binary family constructed in Theorem~\ref{thm:delta-two-bwt}, the BWT
reduces $\delta$ from $\Theta(3^d/d)$ to $2$.  Thus, as previously shown for
the BBWT~\cite{BannaiINakashima2025}, the ordinary BWT can transform
inputs that are asymptotically maximally incompressible under dictionary
measures into highly compressible outputs.

 We expect that different research communities, including those focused on combinatorics on words, commutative algebra, combinatorial algorithms, finite fields, and cryptography, will build upon the new constructions introduced in this paper.
For instance, it would be interesting to develop combinatorial characterizations of the newly introduced classes of words.  Concerning $\delta$, a natural question is how small $\delta(\BWT([w]))$ can be, for each length and alphabet size, when $[w]$ ranges over generalized de Bruijn words.  Another compelling direction is to gain a deeper understanding of the group operation on invertible necklaces and investigate its relationship to the one recently introduced by Zamboni in the context of Christoffel words~\cite{AAM2026_103081}, and further explored in~\cite{ReSh26}.

\bibliography{ref}
\end{document}